\documentclass[11 pt]{amsart}
\usepackage{latexsym,amscd,amssymb, graphicx, shuffle, amsthm}
\usepackage{xcolor}
\usepackage{young}
\usepackage{tikz}
\usepackage{hyperref}
\usepackage[margin=1in]{geometry}
\usepackage{cancel}

\newcommand{\arxiv}[1]{\href{http://arxiv.org/abs/#1}{\texttt{arXiv:#1}}}

\numberwithin{equation}{section}

\newtheorem{theorem}{Theorem}[section]
\newtheorem{proposition}[theorem]{Proposition}
\newtheorem{corollary}[theorem]{Corollary}
\newtheorem{lemma}[theorem]{Lemma}

\theoremstyle{definition}
\newtheorem{defn}[theorem]{Definition}

\newcommand{\maj}{{\mathrm {maj}}}

\newcommand{\Des}{{\mathrm {Des}}}

\newcommand{\grFrob}{{\mathrm {grFrob}}}
\newcommand{\des}{{\mathrm {des}}}

\newcommand{\SYT}{{\mathrm {SYT}}}

\newcommand{\Frob}{{\mathrm {Frob}}}

\newcommand{\symm}{{\mathfrak{S}}}

\newcommand{\CC}{{\mathbb {C}}}
\newcommand{\QQ}{{\mathbb {Q}}}

\newcommand{\OP}{{\mathcal{OP}}}

\newcommand{\AAA}{{\mathcal{A}}}

\newcommand{\GS}{{\mathcal{GS}}}

\newcommand{\xx}{{\mathbf {x}}}
\newcommand{\II}{{\mathbf {I}}}
\newcommand{\yy}{{\mathbf {y}}}
\newcommand{\TT}{{\mathbf {T}}}

\newcommand{\pib}{{ \overline{\pi}}}

\definecolor{darkred}{rgb}{0.7,0,0} % darkred color
\newcommand{\defterm}[1]{{\color{darkred}\emph{#1}}} % emphasis of a definition

\begin{document}

\title[Hall-Littlewood polynomials and a Hecke action on ordered set partitions]
{Hall-Littlewood polynomials and a Hecke action on ordered set partitions}

\author{Jia Huang}
\address
{Department of Mathematics \newline \indent
University of Nebraska at Kearney \newline \indent
Kearny, NE, 68849, USA}
\email{huangj2@unk.edu}

\author{Brendon Rhoades}
\address
{Department of Mathematics \newline \indent
University of California, San Diego \newline \indent
La Jolla, CA, 92093, USA}
\email{bprhoades@math.ucsd.edu}

\author{Travis Scrimshaw}
\address
{School of Mathematics and Physics \newline \indent
University of Queensland \newline \indent
St. Lucia, QLD 4702, Australia}
\email{tcscrims@gmail.com}

\begin{abstract}
We construct an action of the Hecke algebra $H_n(q)$ on a quotient of the 
polynomial ring $F[x_1, \dots, x_n]$, where $F = \QQ(q)$.  The dimension of our quotient ring is the number of 
$k$-block ordered set partitions of $\{1, 2, \dots, n\}$.  This gives a quantum analog of a construction of 
Haglund--Rhoades--Shimozono and interpolates between their result at $q = 1$ and work of Huang--Rhoades at $q = 0$.
\end{abstract}

\maketitle

\section{Introduction}
\label{Introduction}

In this paper we construct a quantum deformation of a recently introduced module~\cite{HRS} 
over the symmetric group $\symm_n$ 
with connections to the  Delta Conjecture~\cite{HRW} in the theory of Macdonald polynomials.
We define and study a graded module $R_{n,k}^{(q)}$ over the Hecke algebra $H_n(q)$ that interpolates 
between a construction of Haglund--Rhoades--Shimozono~\cite{HRS} at $q = 1$  and
a construction of Huang--Rhoades ~\cite{HR} at $q = 0$.
We describe the graded isomorphism type of $R_{n,k}^{(q)}$ and realize $R_{n,k}^{(q)}$ as a 
quantum-deformed instance of the point-orbit method of Garsia-Procesi~\cite{GarsiaProcesi} for constructing 
$\symm_n$-actions on quotients of polynomial rings.

Let $q$ be a formal parameter; we work over the field $F = \QQ(q)$.  The \defterm{(Iwahori-)Hecke algebra}
$H_n(q)$ is a deformation of the group algebra of the symmetric group $\QQ[\symm_n]$.
It is defined as the $F$-algebra generated by $T_1, T_2, \dots, T_{n-1}$ subject to the  relations
\begin{equation}
\begin{cases}
(T_i + 1)(T_i - q) = 0 & \text{for } 1 \leq i \leq n-1, \\
T_i T_j = T_j T_i & \text{for } |i - j| > i, \\
T_i T_{i+1} T_i = T_{i+1} T_i T_{i+1} & \text{for }1 \leq i \leq n-2.
\end{cases}
\end{equation}
When $q = 1$ this is the Coxeter presentation of the symmetric group algebra $\QQ[\symm_n]$.
The algebra $H_n(q)$ has $F$-dimension $n!$ and a linear basis $\{T_w:w\in\symm_n\}$, where $T_w:=T_{s_1}\cdots T_{s_\ell}$ if $w=s_1\cdots s_\ell$ is a reduced expression. 
For $q$ generic (not zero or a root of unity), as it will be in this paper, the $F$-algebra $H_n(q)$ is semisimple and has irreducible representations indexed by partitions $\lambda \vdash n$.
Hecke algebras naturally arise and play significant roles in many places, such as automorphic forms, combinatorics, quantum groups, and the representation theory of symmetric groups and general linear groups~\cite{BumpHecke}.
Finding Hecke deformations of actions of the symmetric group is a pervasive theme in algebraic 
combinatorics (see, {\it e.g.},~\cite{HeckeCoinv,GelfandModels}).

As an ungraded module, our Hecke deformation can be described using ordered set partitions.
A \defterm{$k$-block ordered set partition of size $n$} is a sequence $(B_1 \mid \cdots \mid B_k)$ of $k$ nonempty 
subsets of $[n] := \{1, \dots, n\}$ such that we have the disjoint union $B_1 \sqcup \cdots \sqcup B_k = [n]$.
Let $\OP_{n,k}$ be the family of $k$-block ordered set partitions of size $n$.
For example, $(2 5 \mid 1 \mid 3 4) \in \OP_{5,3}$.  
We may identify $\OP_{n,n}$ with $\symm_n$.

Let $F[\OP_{n,k}]$ be the $F$-vector space with basis $\OP_{n,k}$.
The algebra $H_n(q)$ acts on $F[\OP_{n,k}]$ by the rule
\begin{equation}
T_i.\sigma = \begin{cases}
q s_i(\sigma) + (q-1) \sigma   & \text{if $i+1$ appears in a block to the left of the block containing $i$ in $\sigma$,} \\
s_i(\sigma)  & \text{if $i+1$ appears in a block to the right of the block containing $i$ in $\sigma$,} \\
q \sigma & \text{if $i+1$ appears in the same block as $i$ in $\sigma$,}
\end{cases}
\end{equation}
for $\sigma \in \OP_{n,k}$.
Here $s_i(\sigma)$ is the ordered set partition obtained by interchanging $i$ and $i+1$ in $\sigma$. 
This interpolates between the natural action of $\symm_n$ on $\QQ[\OP_{n,k}]$ at $q = 1$ and 
a `bubble sorting' action of the \defterm{$0$-Hecke algebra} $H_n(0)$ on $\QQ[\OP_{n,k}]$ at $q = 0$ (see~\cite{HR}).
For example, we have
\begin{align*}
T_1.(25 \mid 1 \mid 34) & = q (15 \mid 2 \mid 34) + (q-1) (25 \mid 1 \mid 34), \\
T_2.(25 \mid 1 \mid 34) & = (35 \mid 1 \mid 24), \\
T_3.(25 \mid 1 \mid 34) & = q (25 \mid 1 \mid 34).
\end{align*}
We also construct a graded refinement of this action of $H_n(q)$.

\section{Results}
\label{Results}

We recall the standard action of $H_n(q)$ on the polynomial ring $F[\xx_n] := F[x_1, \dots, x_n]$.
For $1 \leq i \leq n-1$, the adjacent transposition $s_i$ acts on polynomials 
by swapping $x_i$ and $x_{i+1}$:
\begin{equation}
s_i. f(x_1, \dots, x_i, x_{i+1}, \dots, x_n) = f(x_1, \dots, x_{i+1}, x_i, \dots, x_n).
\end{equation}
The \defterm{divided difference operator} $\partial_i$ acts on $F[\xx_n]$ by the rule
\begin{equation}
\partial_i.f(\xx_n) := \frac{f(\xx_n) - s_i.f(\xx_n)}{x_i - x_{i+1}}.
\end{equation}
The \defterm{isobaric divided difference operator} $\pi_i$ is the operator on $F[\xx_n]$ given by composing
multiplication by $x_i$ with $\partial_i$:
\begin{equation}
\pi_i.f(\xx_n) := \partial_i . [x_i f(\xx_n)].
\end{equation}
We will need a modified version of the $\pi_i$ given by $\pib_i := \pi_i - 1$.
If $m$ is a monomial not containing $x_i$ and $x_{i+1}$, then 
\begin{equation}\label{pib}
\overline\pi_i(mx_i^ax_{i+1}^b)=\left\{\begin{array}{ll}
m(x_i^{a-1}x_{i+1}^{b+1}+x_i^{a-2}x_{i+1}^{b+2}\cdots +x_i^bx_{i+1}^a), & {\rm if}\ a>b,\\
0, &{\rm if}\ a=b,\\
-m(x_i^ax_{i+1}^b + x_i^{a+1}x_{i+1}^{b-1} + \cdots + x_i^{b-1}x_{i+1}^{a+1}), & {\rm if}\ a<b.
\end{array}\right.
\end{equation}

Finally, the generator $T_i$ of $H_n(q)$ acts on $F[\xx_n]$ by
\begin{equation}
T_i. f(\xx_n) := q s_i.f(\xx_n) + (1-q) \pib_i.f(\xx_n).
\end{equation}
A direct computation shows that the $T_i$ satisfy the relations of $H_n(q)$.

Let $\Lambda$ be the algebra of symmetric functions in the variable set $\xx = (x_1, x_2, \dots )$.
For any partition $\lambda \vdash n$, let 
\begin{equation*}
m_{\lambda}(\xx), \quad e_{\lambda}(\xx), \quad h_{\lambda}(\xx), \quad s_{\lambda}(\xx), \quad P_{\lambda}(\xx;q)
\end{equation*}
be the associated \defterm{monomial}, \defterm{elementary}, \defterm{(complete) homogeneous}, \defterm{Schur}, and \defterm{Hall-Littlewood $P$-function};
we refer the reader to~\cite{Macdonald} for their definitions.
We will use $m_{\lambda}(\xx_n), e_{\lambda}(\xx_n),$ etc. to denote the restriction of these symmetric functions 
to the variables $\xx_n = (x_1, \dots, x_n)$.
The following quotient ring is the main object of study in this paper.

\begin{defn}
\label{main-definition}
Let $k \leq n$ be positive integers.  Let $I_{n,k}^{(q)} \subseteq F[\xx_n]$ be the ideal
\begin{equation*}
I_{n,k}^{(q)} := \langle P_k(x_1;q), P_k(x_1, x_2; q), \dots, P_k(x_1, x_2, \dots, x_n; q), e_n(\xx_n), e_{n-1}(\xx_n),
\dots, e_{n-k+1}(\xx_n) \rangle
\end{equation*}
and let 
\begin{equation*}
R_{n,k}^{(q)} := F[\xx_n]/I_{n,k}^{(q)}
\end{equation*}
be the corresponding quotient.
\end{defn}

When $q = 1$ we have $P_k(x_1, \dots, x_i; 1) = x_1^k + \cdots + x_i^k$, so that 
\begin{align*}
I_{n,k}^{(1)} &= 
\langle x_1^k, x_1^k + x_2^k, \dots, x_1^k + x_2^k + \cdots + x_n^k, e_n(\xx_n), e_{n-1}(\xx_n),
\dots, e_{n-k+1}(\xx_n) \rangle \\
&=
\langle x_1^k,x_2^k, \dots, x_n^k, e_n(\xx_n), e_{n-1}(\xx_n),
\dots, e_{n-k+1}(\xx_n) \rangle 
\end{align*}
reduces to the ideal $I_{n,k} \subseteq \QQ[\xx_n]$ constructed by Haglund--Rhoades--Shimozono~\cite{HRS}.
When $q = 0$, we have $P_k(x_1, \dots, x_i; 0) = h_k(x_1, \dots, x_i)$, so that
\begin{equation*}
I_{n,k}^{(0)} =
 \langle h_k(x_1), h_k(x_1, x_2), \dots, h_k(x_1, x_2, \dots, x_n), e_n(\xx_n), e_{n-1}(\xx_n), \dots, e_{n-k+1}(\xx_n) \rangle
\end{equation*}
is the ideal $J_{n,k}$ studied by Huang--Rhoades in the context of the 0-Hecke algebra $H_n(0)$~\cite{HR}.

The ideal $I_{n,k}^{(q)}$ is homogeneous, so $R_{n,k}^{(q)}$ has the structure of a graded $F$-vector space.
When $k = n$, the ideal $I_{n,n}^{(q)}$ is the classical \defterm{invariant ideal} generated by the space $F[\xx_n]^{\symm_n}_+$
of $\symm_n$-invariant polynomials with vanishing constant term, 
and $R_{n,n}^{(q)} = F[\xx_n]/\langle F[\xx_n]^{\symm_n}_+ \rangle$ is the classical \defterm{coinvariant algebra} $R_n$ of the symmetric group $\symm_n$.  
There is a well-known action of the Hecke algebra $H_n(q)$ on $R_n$~\cite{HeckeCoinv}; we prove that the same is true for $R_{n,k}^{(q)}$.

\begin{proposition}
\label{r-is-stable}
Let $k \leq n$ be positive integers.  The ideal $I_{n,k}^{(q)}$ is stable under the action of $H_n(q)$ on $F[\xx_n]$.
The quotient ring $R_{n,k}^{(q)}$ is therefore a graded $H_n(q)$-module.
\end{proposition}

The proof of Proposition~\ref{r-is-stable} is a slightly tedious computation and will be postponed to Section~\ref{Proofs}.

Since our Hecke parameter $q$ is generic, the $F$-algebra $H_n(q)$ is semisimple with irreducible representations
$W^{\lambda}$ indexed by partitions $\lambda \vdash n$~\cite{Macdonald}.  
If $V$ is any finite-dimensional $H_n(q)$-module, there 
exist unique nonnegative integers $c_{\lambda}$ such that $V \cong_{H_n(q)} \bigoplus_{\lambda \vdash n} c_{\lambda} W^{\lambda}$.
The \defterm{Frobenius image} of $V$ is the symmetric function
$\Frob(V) := \sum_{\lambda \vdash n} c_{\lambda} s_{\lambda}(\xx)$.
More generally, if $V = \bigoplus_{d \geq 0} V_d$ is a {\em graded} $H_n(0)$-module with each graded piece $V_d$ 
finite-dimensional, the \defterm{graded Frobenius image} of $V$ is $\grFrob(V;t) := \sum_{d \geq 0} \Frob(V_d) \cdot t^d$.
We wish to determine $\grFrob(R_{n,k}^{(q)}; t)$; to do this, we employ
a quantum deformation of the point-orbit method.

Pioneered by Garsia-Procesi~\cite{GarsiaProcesi} in the context of the Tanisaki ideals, the \defterm{point-orbit method}
gives a systematic way for producing interesting graded modules over the symmetric group $\symm_n$
(or more generally any finite matrix group $G$) from a finite set $Y$ of points in an $n$-dimensional space
 which is closed under the group action.  
 We recall their construction.
 
 Let $Y \subseteq F^n$ be a finite point set.
We denote by
 \begin{equation}
 \II(Y) := \{ f \in F[\xx_n] \,:\, f(\yy) = 0 \text{ for all $\yy \in Y$} \}
 \end{equation}
the ideal of polynomials vanishing on $Y$.  
The ideal $\II(Y)$ is usually not homogeneous; to produce a 
 homogeneous ideal we consider
 \begin{equation}
 \TT(Y) := \langle \tau(f) \,:\, f \in \II(Y) - \{ 0 \} \rangle.
 \end{equation}
Here if $f \in F[\xx_n]$ is any nonzero polynomial and $f = f_d + \cdots  + f_1 + f_0$, where $f_i$ is homogeneous
 of degree $i$ and $f_d \neq 0$, we let $\tau(f) := f_d$ be the highest degree component of $f$.  The ideal $\TT(Y)$ is 
 homogeneous by definition. We have
 \begin{equation}
 |Y| = \dim_F(F[\xx_n]/\II(Y)) = \dim_F(F[\xx_n]/\TT(Y)).
 \end{equation}
 
 Let $G \subseteq GL_n(F)$ be a finite matrix group.  The group $G$ acts on $F[\xx_n]$ by linear substitutions.
 If a finite point set $Y \subseteq F^n$ is closed under the action of $G$, we have isomorphisms of $G$-modules
 \begin{equation}
 F[Y] \cong_G F[\xx_n]/\II(Y) \cong_G F[\xx_n]/\TT(Y),
 \end{equation}
 where $F[Y]$ is the permutation representation of $G$ on $Y$.
 This has been used to produce a number of interesting graded $G$-modules:
 \begin{itemize}
 \item
 For $G = \symm_n$, and $Y$ a single $\symm_n$-orbit in $\QQ^n$, 
 Garsia--Procesi showed that $\TT(Y)$ is the \defterm{Tanisaki ideal} which governs the cohomology of a 
 Springer fiber~\cite{GarsiaProcesi}.
 \item
For $G = \symm_n$, the point-orbit method was used in~\cite{HRS} to study the ring $R_{n,k}^{(1)}$.  
After fixing distinct field elements $\alpha_1, \dots , \alpha_k \in F$,
the point set $Y$ in this case is
\[
Y = \bigl\{ (y_1, \dots, y_n) \in F^n \,:\, \{y_1, \dots, y_n\} = \{\alpha_1, \dots, \alpha_k \} \bigr\}.
\]
There is an evident bijection between $Y$ and $\OP_{n,k}$.
\item
Let $r \geq 2$ and $G = G(r,1,n)$, the group of $n \times n$ monomial matrices whose nonzero entries are $r^{th}$ roots
of unity in $\CC$.  For $k \leq n$, Chan--Rhoades~\cite{CR} used the point-orbit method (over the field $\CC$) to produce a quotient
$R_{n,k}^G$ of $\CC[\xx_n]$ whose dimension equals the number of $k$-dimensional faces in the Coxeter complex attached to 
$G$.  If $\alpha_1, \dots, \alpha_k$ are distinct positive real numbers, the point set $Y$ is 
\begin{equation*}
Y = \{ (y_1, \dots, y_n) \in \CC^n \,:\, \text{ $\{y_1^r, \dots, y_n^r \} = \{\alpha_1, \dots, \alpha_k \}$ or 
$\{0, \alpha_1, \dots, \alpha_k \}$} \}.
\end{equation*}
\end{itemize}

In the work of Huang--Rhoades on an action of the 0-Hecke algebra on ordered set partitions~\cite{HR}, the acting algebraic
object was not a group but rather the 0-Hecke algebra $H_n(0)$.  Despite this, in~\cite{HR} it is proven that if we let 
$Y$ be the point set 
\begin{equation*}
Y = \bigl\{ (y_1, \dots, y_n) \in \mathbb{F}^n \,:\,
\text{$y_1, \dots, y_n$ distinct}, 
\{\alpha_1, \dots, \alpha_k \} \subseteq \{y_1, \dots, y_n \}, 
y_i \in \{ \alpha_1, \dots, \alpha_{k+i-1} \} \bigr\}
\end{equation*}
where $\mathbb{F}$ is an arbitrary field and
$\alpha_1, \alpha_2, \dots, \alpha_{n+k-1} \in \mathbb{F}$ are distinct field elements, 
the ideal $\TT(Y)$ is closed under the action of  $H_n(0)$ on $\mathbb{F}[\xx_n]$ by isobaric
divided difference operators and the quotient $\mathbb{F}[\xx_n]/\TT(Y)$ may be identified with a
natural 0-Hecke action on $\mathbb{F}[\OP_{n,k}]$.
The point locus $Y$ used in this paper is as follows.

\begin{defn}
\label{point-set-definition}
Let $\alpha_1, \dots, \alpha_k \in \QQ$ be distinct rational numbers.  Let $Y_{n,k}^{(q)} \subseteq F^n$ be the set of points
$(y_1, \dots, y_n)$ such that
\begin{itemize}
\item  $y_i \in \{ q^j \cdot \alpha_r \,:\, j \geq 0, 1 \leq r \leq k \}$ for all $1 \leq i \leq n$,
\item  the coordinates $y_1, \dots, y_n$ are distinct,
\item  $\{ \alpha_1, \dots, \alpha_k \} \subseteq \{y_1, \dots, y_n \}$, and
\item  if $y_i = q^j \cdot \alpha_r$ for some $1 \leq i \leq n$,  $j > 0$, and $1 \leq r \leq k$, there exists $i' < i$ such that
$y_{i'} = q^{j-1} \cdot \alpha_r$.
\end{itemize} 
\end{defn}

The point set $Y_{n,k}^{(q)}$ is in bijective correspondence with $\OP_{n,k}$.  Indeed, given 
$\sigma = (B_1 \mid \cdots \mid B_k)$, we have a point $\varphi(\sigma) = (y_1, \dots, y_n) \in Y_{n,k}$ 
given by the rule $y_i = q^j \cdot \alpha_r$ if $i$ is the $(j+1)^{st}$-smallest letter in the block $B_r$ of $\sigma$.
As an example, we have 
\begin{equation*}
\varphi \colon ( 5  \mid 1 4 6 \mid 2 3 ) \mapsto (\alpha_2, \alpha_3, q \cdot \alpha_3, q \cdot \alpha_2, \alpha_1, q^2 \cdot \alpha_2).
\end{equation*}
Definition~\ref{point-set-definition} is designed so that the map $\varphi \colon \OP_{n,k} \rightarrow Y_{n,k}^{(q)}$ is a bijection.

When $q = 1$, the point set $Y_{n,k}^{(1)}$ gives the labeling of ordered set partitions used to study the
$\symm_n$-module $R_{n,k}^{(1)}$ in~\cite{HRS}.  When $q = 0$, the point set $Y_{n,k}^{(0)}$ becomes `degenerate'
when $k < n$; there are fewer points in $Y_{n,k}^{(0)}$ than there are ordered set partitions in $\OP_{n,k}$.
The point set in~\cite{HR} used to study the 0-Hecke structure of $R_{n,k}^{(0)}$ looks very different from $Y_{n,k}^{(0)}$.
As it turns out, the point set $Y_{n,k}^{(q)}$ gives rise to the quotient $R_{n,k}^{(q)}$ and in this way is a quantum deformation
of the point set used in~\cite{HRS}.

\begin{theorem}
\label{point-set-theorem}
Let $k \leq n$ be positive integers.  We have $\TT(Y_{n,k}^{(q)}) = I_{n,k}^{(q)}$, so that we have the identification of quotients
$F[\xx_n]/\TT(Y_{n,k}^{(q)}) = R_{n,k}^{(q)}$.
\end{theorem}

We will prove Theorem~\ref{point-set-theorem} in Section~\ref{Proofs}.  

The quantum deformation involved in Definition~\ref{point-set-definition}
can be used to define graded $H_n(q)$-modules in a broader context.  In particular,
let $Y \subset \QQ^n$ be any finite point set which is closed under the action of 
$\symm_n$.
Given any point $y = (y_1, \dots, y_n) \in Y$, consider the point
$y^{(q)} = (y_1^{(q)}, \dots, y_n^{(q)}) \in F^n$ given by
$y_i^{(q)} = q^{j-1} \cdot y_i$, where $i$ is the $j^{th}$ occurrence of the rational
number $y_i$ in the list $(y_1, \dots, y_n)$.
As an example, if $\alpha_1, \alpha_2, \alpha_3 \in \QQ$ are distinct rational numbers
and $y = (\alpha_2, \alpha_1, \alpha_2, \alpha_2, \alpha_3, \alpha_1)$, then
$y^{(q)} = (\alpha_2, \alpha_1, q \cdot \alpha_2, q^2 \cdot \alpha_2, \alpha_3, q \cdot \alpha_1)$.
We consider the new point set $Y^{(q)} := \{ y^{(q)} \,:\, y \in Y \} \subset F^n$, so that 
(for example) if
\begin{equation*}
Y = \{ (\alpha_1, \alpha_1, \alpha_2), 
(\alpha_1, \alpha_2, \alpha_1), 
(\alpha_2, \alpha_1, \alpha_1) \},
\end{equation*}
for $\alpha_1 \neq \alpha_2$, then
\begin{equation*}
Y^{(q)} = \{ (\alpha_1, q \cdot \alpha_1, \alpha_2),
(\alpha_1, \alpha_2, q \cdot \alpha_1),
(\alpha_2, \alpha_1, q \cdot \alpha_1) \}.
\end{equation*}

The point set $Y$ can be recovered from the point set $Y^{(q)}$ by setting $q = 1$, but
the quantization $Y \leadsto Y^{(q)}$ has the effect of breaking $\symm_n$-symmetry;
the homogeneous ideal $\TT(Y^{(q)}) \subset F[\xx_n]$ is usually not $\symm_n$-stable.
On the other hand, Kyle Meyer (personal communication) proved that 
$\TT(Y^{(q)})$ {\em is} stable under the action of $H_n(q)$ on 
$F[\xx_n]$, so that $F[\xx_n]/\TT(Y^{(q)})$ is a graded $H_n(q)$-module
of dimension $|Y|$.
Therefore, any graded $\symm_n$-module constructed by the point-orbit
method has a natural companion graded $H_n(q)$-module.
Although there are examples where the graded dimensions (i.e. the Hilbert series)
of $\QQ[\xx_n]/\TT(Y)$ and $F[\xx_n]/\TT(Y^{(q)})$ are different, this
does not happen in our context; to prove this, we will describe the 
Gr\"obner basis of $I_{n,k}^{(q)}$.

The Gr\"obner theory of the ideal $I_{n,k}^{(q)}$ is a straightforward $q$-analog of the corresponding theory for $I_{n,k}$.
We consider the term order $<$ on monomials in $F[\xx_n]$ given by $x_1^{a_1} \cdots x_n^{a_n} < x_1^{b_1} \cdots x_n^{b_n}$
if there exists $1 \leq i \leq n$ with $a_i < b_i$ and $a_{i+1} = b_{i+1}, \dots, a_n = b_n$.  Following the notation of
$\textsc{SageMath}$~\cite{sage}, we call this term order $\texttt{neglex}$.

Recall that a \defterm{shuffle} of two sequences $(a_1, \dots, a_r)$ and $(b_1, \dots, b_s)$ is an interleaving
$(c_1, \dots, c_{r+s})$ of these sequences which preserves the relative order of the $a$'s and the $b$'s.  
An \defterm{$(n,k)$-staircase} is a shuffle of the sequences $(k-1, \dots, 1, 0)$ and
$(k-1, k-1, \dots, k-1)$, where the second sequence has $n-k$ copies of $k-1$.  For example, the $(5,3)$-staircases are
\begin{equation*}
(2,2,2,1,0), (2,2,1,2,0), (2,2,0,1,2), (2,1,2,2,0), (2,1,2,0,2), \text{ and } (2,1,0,2,2).
\end{equation*}
The \defterm{$(n,k)$-Artin monomials} $\AAA_{n,k}$ are those monomials $x_1^{a_1} \cdots x_n^{a_n}$ in the variables 
$x_1, \dots, x_n$ whose exponent sequences $(a_1, \dots, a_n)$ are componentwise $\leq$ some $(n,k)$-staircase.
These are `reverse to' the $(n,k)$-Artin monomials as defined in~\cite{HRS}.

If $\gamma = (\gamma_1, \dots, \gamma_n)$ is a weak composition with $n$ parts, we let $\kappa_{\gamma}(\xx_n) \in F[\xx_n]$
be the corresponding \defterm{Demazure character}; see~\cite{HRS} for its definition.
If $S  = \{s_1 < \cdots < s_r \} \subseteq [n]$ is any subset, we consider the \defterm{skip composition} 
$\gamma(S) = (\gamma_1, \dots, \gamma_n)$ defined by
\begin{equation*}
\gamma_i = \begin{cases}
s_j - j + 1 & \text{if $i = s_j \in S$,} \\
0 & \text{if $i \notin S$.}
\end{cases}
\end{equation*}
We will also need the reverse $\gamma(S)^* = (\gamma_n, \dots, \gamma_1)$ of this weak composition.
As an example, if $n = 6$ and $S = \{2,5,6\}$ then $\gamma(S) = (0,2,0,0,4,4)$ and $\gamma(S)^* = (4,4,0,0,2,0)$.

\begin{corollary}
\label{groebner-corollary}
Let $k \leq n$ be positive integers and
give monomials in $F[\xx_n]$ the term order $\texttt{neglex}$.  
The standard monomial basis for the ideal $I_{n,k}^{(q)}$ is the set $\AAA_{n,k}$ of $(n,k)$-Artin monomials.
A Gr\"obner basis for the ideal $I_{n,k}^{(q)}$ is given by
the Hall-Littlewood $P$-functions
\begin{equation*}
P_k(x_1;q), P_k(x_1, x_2; q), \dots, P_k(x_1, x_2, \dots, x_n; q)
\end{equation*}
together with 
the  Demazure characters
\begin{equation*}
\kappa_{\gamma(S)^*}(\xx_n) \text{ for all } S \subseteq [n-1] \text{ such that } |S| = n-k+1.
\end{equation*}
If $k < n$, this Gr\"obner basis is minimal.  
\end{corollary}

\begin{proof}
It is well known (see, {\it e.g.},~\cite[Lem. 3.5]{HRS}) that the $\texttt{neglex}$-leading monomial of the 
Demazure character $\kappa_{\gamma(S)^*}(\xx_n)$ is $x_1^{\gamma_n} \cdots x_{n-1}^{\gamma_2} x_n^{\gamma_1}$
for any subset $S \subseteq [n-1]$ with $|S| = n-k+1$ and  $\gamma(S) = (\gamma_1, \dots, \gamma_n)$.
The $\texttt{neglex}$-leading monomial of the polynomial $P_k(x_1, \dots, x_i; q)$ is the variable power $x_i^k$.

By~\cite[Thm. 4.9]{HRS}, the number of monomials in $F[\xx_n]$ which are not divisible by any of the leading
monomials in the above paragraph equals $|\OP_{n,k}|$.   
Moreover,~\cite[Lem 3.4]{HRS} 
(and in particular~\cite[Eqn. 3.4]{HRS}) shows that the Demazure characters appearing in the statement of the corollary
actually lie in the ideal $I_{n,k}^{(q)}$, so that the leading monomials in the above paragraph are actually all leading monomials
of polynomials in $I_{n,k}^{(q)}$.  Theorem~\ref{point-set-theorem} implies $\dim(R_{n,k}^{(q)}) = |\OP_{n,k}|$,
so that the polynomials in the statement of the present corollary form a Gr\"obner basis of $I_{n,k}^{(q)}$.
The statement about minimality when $k < n$ comes from the forms of the leading monomials in the above paragraph.

The claim about the standard monomial basis of $R_{n,k}^{(q)}$ being $\AAA_{n,k}$ is obtained from~\cite[Thm. 4.13]{HRS}
by reversing the variable order $(x_1, \dots, x_n) \leadsto (x_n, \dots, x_1)$.
\end{proof}

Given a permutation $w \in \symm_n$ with one-line notation $w = w_1 \dots w_n$, the associated \defterm{Garsia-Stanton
monomial} is $gs_w := \prod_{w_i > w_{i+1}} x_{w_1} x_{w_2} \cdots x_{w_i}$.  Since $P_k(x_1, \dots, x_i; q)$
has the form
\begin{equation*}
P_k(x_1, \dots, x_i; q) = x_i^k + \text{other homogeneous degree $k$ terms involving $x_1, \dots, x_i$},
\end{equation*}
the proof of~\cite[Lem. 4.2]{HR} (see also~\cite[Cor. 4.3]{HR})
goes through directly to show that the \defterm{generalized GS monomials}
\begin{equation}
\GS_{n,k} := \{ gs_w \cdot x_{w_1}^{i_1} x_{w_2}^{i_2} \cdots x_{w_{n-k}}^{i_{n-k}} \,:\, 
w \in \symm_n,  k - \des(w) > i_1 \geq i_2 \geq \cdots \geq i_{n-k} \geq 0 \}
\end{equation}
also descend to a basis for $R_{n,k}^{(q)}$; we omit the details.

Let $\SYT(n)$ be the set of standard Young tableaux with $n$ boxes.  For a tableau $T \in \SYT(\lambda)$, let 
$\mathrm{sh}(T) \vdash n$ be the partition given by the shape of $T$.  An index $1 \leq i \leq n-1$ is a \defterm{descent}
of $T$ if $i$ appears above $i+1$ in $T$ (drawn in the English notation).  Let $\Des(T)$ be the set of all descents of $T$, let $\des(T) := |\Des(T)|$
be the number of descents of $T$, and let $\maj(T) := \sum_{i \in \Des(T)} i$ be the \defterm{major index} of $T$.
We use the following $t$-analogs of numbers, factorials, and binomial coefficients:
\begin{equation}
[n]_t := 1 + t + \cdots + t^{n-1}, \qquad [n]!_t := [n]_t [n-1]_t \cdots [1]_t, \qquad {n \brack k}_t := \frac{[n]!_t}{[k]!_t [n-k]!_t},
\end{equation}
with the understanding that ${n \brack k}_t = 0$ if $n < k$ or $k < 0$.

\begin{corollary}
\label{graded-frobenius-image}
Let $k \leq n$ be positive integers.  The graded Frobenius image of the $H_n(q)$-module
$R_{n,k}^{(q)}$ has Schur expansion
\begin{equation}
\grFrob(R_{n,k}^{(q)}; t) = \sum_{T \in \SYT(n)} t^{\maj(T)} {n - \des(T) - 1 \brack n-k}_t s_{\mathrm{sh}(T)}(\xx).
\end{equation}
%where $\lambda'$ denotes the conjugate of a partition $\lambda$.
\end{corollary}

\begin{proof}
For $\lambda \vdash n$, let $S^{\lambda}$ be the corresponding irreducible representation of the symmetric group
$\symm_n$.
If we fix a basis for any $H_n(q)$-irreducible $W^{\lambda}$, it is well known 
(see, {\it e.g.},~\cite[Thm. 8.1.7]{GeckPfeiffer})
that the representing matrix for any 
generator $T_i$ acting on $W^{\lambda}$ specializes to a representing matrix for the adjacent transposition
$s_i = (i, i+1)$ acting on the corresponding $\symm_n$-irreducible $S^{\lambda}$ at $q = 1$.

By Proposition~\ref{r-is-stable}, the quotient $R_{n,k}^{(q)}$ is a graded $H_n(q)$-module.  Since $H_n(q)$
is semisimple, for any degree $d$ there exist unique integers $c_{\lambda} \geq 0$ such that the $d^{th}$
graded piece $(R_{n,k}^{(q)})_d$ decomposes into irreducibles as 
$(R_{n,k}^{(q)})_d \cong_{H_n(q)} \bigoplus_{\lambda \vdash n} c_{\lambda} W^{\lambda}$.
By Corollary~\ref{groebner-corollary} and \cite[Thm. 4.14]{HRS}, 
the modules $R_{n,k}^{(q)}$ and $R_{n,k}^{(1)}$ have the same Hilbert series,
so that $\dim (R_{n,k}^{(q)})_d = \dim (R_{n,k}^{(1)})_d$.  We may therefore consider
the representing matrices for the action of $\symm_n$ on $(R_{n,k}^{(1)})_d$
(defined over $\QQ$) as the $q = 1$ specializations of the corresponding representing
matrices (defined over $F$) for the action of $H_n(q)$ on $(R_{n,k}^{(q)})_d$.
The above paragraph implies that we have a $\symm_n$-module decomposition
$(R_{n,k}^{(1)})_d \cong_{\symm_n} \bigoplus_{\lambda \vdash n} c_{\lambda} S^{\lambda}$
involving the same multiplicities $c_{\lambda}$.
The result follows from the calculation of $\grFrob(R_{n,k}^{(1)}; t)$ in~\cite{HRS}.
\end{proof}

For example, we have
\[ \grFrob(R_{4,2}^{(q)}; t) =  t^0 {3\brack 2}_t s_{(4)}(\xx)
+ ( t^3 + t^2 + t^1 ) {2\brack 2}_t s_{(3,1)}(\xx) 
+ t^2 {2\brack 2}_t s_{(2,2)}(\xx). \]
Adin, Brenti, and Roichman \cite{ABR} studied a refinement of the classical coinvariant ring
$R_n = R_{n,n}$ indexed by all possible partitions 
$\lambda$ with $\leq n$ parts which is finer than the degree grading and
whose module structure is governed
by descent sets of  tableaux $T \in \SYT(n)$.
Meyer \cite[Thm. 1.4]{Meyer} extended this result to the quotients $R_{n,k}$ for $k \leq n$.
It may be interesting to refine Corollary~\ref{graded-frobenius-image} to 
obtain a quantum analog of Meyer's results.

The quotient $R_{n,k}^{(q)}$ gives a graded refinement of the action of $H_n(q)$ on $F[\OP_{n,k}]$.

\begin{corollary}
\label{ungraded-frobenius-image}
Let $k \leq n$ be positive integers.  We have $R_{n,k}^{(q)} \cong_{H_n(q)} F[\OP_{n,k}]$ as ungraded $H_n(q)$-modules.
\end{corollary}

\begin{proof}
The argument given in the proof of Corollary~\ref{graded-frobenius-image} reduces us to proving the $q = 1$ specialization $R_{n,k}^{(1)}$
is isomorphic as an $\symm_n$-module to the standard permutation action of $\symm_n$ on $\QQ[\OP_{n,k}]$;
this was accomplished in~\cite{HRS}.
\end{proof}

The statements of
Corollaries~\ref{graded-frobenius-image} and \ref{ungraded-frobenius-image} interpolate between results of~\cite{HRS}
at $q = 1$ and~\cite{HR} at $q = 0$.  Since the $0$-Hecke algebra $H_n(0)$ is not semisimple, the proofs of these 
corollaries do not go through to give the corresponding $q = 0$ results of~\cite{HR}.

\section{Proofs}
\label{Proofs}

The Hall-Littlewood polynomials $P_d(x_1, \dots, x_i; q)$ have the following generating function
(see~\cite[p. 209]{Macdonald}), which we take as the definition of $P_d(x_1, \dots, x_i; q)$:
\begin{equation}
\label{macdonald-generating-function}
\sum_{d \geq 0} (1-q) P_d(x_1, \dots, x_i; q) \cdot t^d = \prod_{j = 1}^i \frac{1 - q x_j t}{1 - x_j t}.
\end{equation}
We will need the following well known expansion of $P_d(x_1, \dots, x_i; q)$ into the monomial basis of
symmetric functions.

\begin{lemma}
\label{p-to-m-lemma}
For any nonnegative integer $d$, the Hall-Littlewood polynomial $P_d(\xx_i;u)$ expands in terms of the
monomial symmetric functions $m_{\lambda}(\xx_i)$ as
\begin{equation}
P_d(x_1, \dots, x_i ; q) = \sum_{\lambda \vdash d} (1-q)^{\ell(\lambda) - 1} m_{\lambda}(x_1, \dots, x_i).
\end{equation}
\end{lemma}

\begin{proof}
Starting with Equation~\eqref{macdonald-generating-function} we have
\begin{equation}
\begin{split}
\sum_{d \geq 0} (1-q) P_d(x_1, \dots, x_i; q) \cdot t^d &= \prod_{j = 1}^i \frac{1 - q x_j t}{1 - x_j t} \\
&= \prod_{j = 1}^i \left[ \frac{1-q}{1- x_j t} +  q \right] \\
&= \prod_{j = 1}^i [1 + (1-q) x_j t + (1-q) x_j^2 t^2 + \cdots ] \\
&= \sum_{d \geq 0} \sum_{\lambda \vdash n} (1-q)^{\ell(\lambda)} m_{\lambda}(x_1, \dots, x_i) t^d.
\end{split}
\end{equation}
Dividing both sides by $(1-q)$ and taking the coefficient of $t^k$ gives the result.
\end{proof}

We will also need the following version of the Leibniz rule for the action of the generator $T_i$ of the Hecke algebra $H_n(q)$ on products of polynomials in $F[\xx_n]$.

\begin{lemma}
\label{leibniz-lemma}
Let $1 \leq i \leq n-1$ and $f, g \in F[\xx_n]$.  We have
\begin{equation}
T_i.(fg) = (s_i.f)(T_i.g) + (1-q) (\pib_i.f)g.
\end{equation}
\end{lemma}

\begin{proof}
Using the `Leibniz' Rule $\pib_i.(fg) = (\pib_i.f)g + (s_i.f)(\pib_i.g)$ we have
\begin{align*}
T_i.(fg) &= q(s_i.f)(s_i.g) + (1-q) (\pib_i.f)g + (1-q) (s_i.f)(\pib_i.g) \\
&= (s_i.f)(T_i.g) + (1-q) (\pib_i.f)g.
\end{align*}
\end{proof}

We are ready to prove Proposition~\ref{r-is-stable}.

\begin{proof}[Proof of Proposition~\ref{r-is-stable}]
Let $1 \leq i \leq n-1$.
By Lemma~\ref{leibniz-lemma}, it suffices to show that for any generator $g$ of $I_{n,k}^{(q)}$ we have 
$T_i.g \in I_{n,k}^{(q)}$.  For then if $f \in F[\xx_n]$ is arbitrary we have
$T_i.(fg) = (1-q) x_{i+1} (\partial_i.f) g + (s_i.f)(T_i.g)$, and both terms on the right hand side lie in 
$I_{n,k}^{(q)}$. 

If $g \in F[\xx_n]$ is symmetric in $x_i, x_{i+1}$ then $\overline{\pi}_i.g = 0$ so that 
$T_i.g = q(s_i.g) = qg$.  The generators $e_n(\xx_n), e_{n-1}(\xx_n), \dots, e_{n-k+1}(\xx_n)$, as well
as the generators $P_k(x_1, \dots, x_j; q)$ for $j \neq i$, are symmetric in $x_i, x_{i+1}$ so that if $g$ is any 
of these generators we have $T_i.g = qg \in I_{n,k}^{(q)}$.  

We are reduced to proving $T_i.P_k(x_1, \dots, x_i; q) \in I_{n,k}^{(u)}$.  
This would be a consequence of 
\begin{equation}
\label{main-closure-equation}
T_i.P_k(x_1, \dots, x_i; q) = P_k(x_1, \dots, x_{i+1}; q) - P_k(x_1, \dots, x_i; q) + q P_k(x_1, \dots, x_{i-1}; q),
\end{equation}
where we interpret $P_k(x_1, \dots, x_{i-1}; q)$ to be $0$ when $i = 1$,
since the right hand side of Equation~\eqref{main-closure-equation} clearly lies in $I_{n,k}^{(q)}$.

To prove Equation~\eqref{main-closure-equation} we compare the coefficients of monomials $m$ on both sides, making
use of Lemma~\ref{p-to-m-lemma}.  
Our analysis 
breaks up into four cases depending on whether $x_i$ or $x_{i+1}$ appears in $m$ with a positive exponent.  Let $p$ 
be the total number of variables appearing in $m$ with positive exponent.

{\bf Case 1:}  {\em Neither $x_i$ nor $x_{i+1}$ appears in $m$.} 
The term corresponding to $m$ on the left hand side of Equation~\eqref{main-closure-equation}
is $(1-q)^{p-1} T_i.m = q (1-q)^{p-1} m$.  The term corresponding to $m$ on the right hand side of 
Equation~\eqref{main-closure-equation} is 
\begin{equation}
(1-q)^{p-1} m - (1-q)^{p-1} m + q(1-q)^{p-1}m = q (1-q)^{p-1} m.
\end{equation}

{\bf Case 2:}  {\em $x_i$ appears in $m$, but $x_{i+1}$ does not.}  Write $m = m' x_i^a$ where $m'$ is a monomial
in $x_1, \dots, x_{i-1}$.  Since the coefficient of $m$ in $\overline{\pi}_i(m)$ is $0$ by
Equation~\eqref{pib}, the monomial $m$ does not 
appear on the left hand side of Equation~\eqref{main-closure-equation}.  The term corresponding to $m$ on the right hand side is
\begin{equation}
(1-q)^{p-1} m - (1-q)^{p-1} m + 0 = 0.
\end{equation}

{\bf Case 3:}  {\em $x_{i+1}$ appears in $m$, but $x_i$ does not.}  Write $m = m' x_{i+1}^b$ where $m'$ is a monomial
in $x_1, \dots, x_{i-1}$.
The coefficient of $m$ on the left hand side of Equation~\eqref{main-closure-equation} is the coefficient of $m$
in $T_i.((1-q)^{p-1} m' x_i^b)$, which is $(1-q)^{p-1}[q + (1-q)] = (1-q)^{p-1}$.  
The term corresponding to $m$ on the right hand side is
$(1-q)^{p-1} m - 0 + 0 = (1-q)^{p-1} m$.

{\bf Case 4:}  {\em Both $x_i$ and $x_{i+1}$ appear in $m$.}  Write $m = m' x_i^a x_{i+1}^b$ where $m'$ is a monomial
in $x_1, \dots, x_{i-1}$.  The coefficient of $m$ on the left hand side of Equation~\eqref{main-closure-equation} is the coefficient
of $m$ in $T_i.((1-q)^{p-2} m' x_i^{a+b})$, which is $(1-q)(1-q)^{p-2} = (1-q)^{p-1}$.  This is also the coefficient of $m$
on the right hand side of Equation~\eqref{main-closure-equation}.
\end{proof}

By either Corollary~\ref{graded-frobenius-image} or Corollary~\ref{ungraded-frobenius-image}, we
know $\dim(R_{n,k}^{(q)}) =  |\OP_{n,k}|$.
This puts us in a good position to prove Theorem~\ref{point-set-theorem}.

\begin{proof}[Proof of Theorem~\ref{point-set-theorem}]
We show that the generators of $I_{n,k}^{(q)}$ arise as highest degree components of certain polynomials in $\II(Y_{n,k}^{(q)})$, where the point set $Y_{n,k}^{(q)}$ is defined in terms of distinct rational numbers $\alpha_1,\ldots,\alpha_k$ satisfying certain conditions (see Definition~\ref{point-set-definition}).
To start, let $n-k+1 \leq d \leq n$.  We claim  
\begin{equation}
\label{first-inclusion}
\sum_{i = 0}^d (-1)^{d-i} e_i(\xx_n) h_{d-i}(\alpha_1, \dots, \alpha_k) \in \II(Y_{n,k}^{(q)}),
\end{equation}
so that taking the top component gives $e_d(\xx_n) \in \TT(Y_{n,k}^{(q)})$.  To see~\eqref{first-inclusion}, notice that the alternating sum
$\sum_{i = 0}^d (-1)^{d-i} e_i(\xx_n) h_{d-i}(\alpha_1, \dots, \alpha_k)$ is the coefficient 
of $t^d$ in the rational function
\begin{equation}
\frac{(1 + x_1 t)(1 + x_2 t) \cdots (1 + x_n t)}{(1 + \alpha_1 t) (1 + \alpha_2 t) \cdots (1 + \alpha_k t)}.
\end{equation}
Since the numbers $\alpha_1, \dots, \alpha_k$ must appear as coordinates of points in $Y_{n,k}^{(q)}$, when 
$(x_1, \dots, x_n) \in Y_{n,k}^{(q)}$ the factors in the denominator cancel with $k$ factors in the numerator, 
yielding a polynomial in $t$ of degree $n-k < d$.
Thus the coefficient of $t^d$ evaluated at $(x_1, \dots, x_n) \in Y_{n,k}^{(q)}$ must be zero.

Next, let $1 \leq i \leq n$.  We need to show $P_k(x_1, \dots, x_i; q) \in \TT(X_{n,k}^{(q)})$.  
To do this, we will use the generating function
for the $P_k(x_1, \dots, x_i; q)$ provided by Equation~\eqref{macdonald-generating-function}.

We claim that
\begin{equation}
\label{i-in-t-main-identity}
\sum_{j = 0}^k (-1)^{k-j} (1-q) P_j(x_1, \dots, x_i; q) e_j(\alpha_1, \dots, \alpha_k) \equiv 
(-1)^k \cdot q^i \cdot \alpha_1 \cdots \alpha_k
 \text{ on $Y_{n,k}^{(q)}$.}
\end{equation}
To see this, notice that by Equation~\eqref{macdonald-generating-function}
the left hand side of Equation~\eqref{i-in-t-main-identity} is the coefficient of $t^k$ in
the expression
\begin{equation}
\label{cancelling-expression}
\left( \prod_{j = 1}^i \frac{1 - q x_j t}{1 - x_j t} \right) \cdot (1 - \alpha_1 t) (1 - \alpha_2 t) \cdots (1 - \alpha_k t).
\end{equation}
Consider a typical ordered set partition $\sigma = (B_1 \mid \cdots \mid B_k) \in \OP_{n,k}$ and the corresponding 
point $\varphi(\sigma) = (y_1, \dots, y_n) \in Y_{n,k}^{(q)}$.   Let $i_j$ be the number of entries in the block $B_j$
of $\sigma$ which are $\leq i$; we have $i_1 + \cdots + i_k = i$.
Upon specialization to $\varphi(\sigma)$, the expression~\eqref{cancelling-expression} equals 
\begin{equation}
\left( \prod_{j = 1}^k  \frac{1 - q^{i_j} \alpha_j t}{1 - \alpha_j t} \right) \cdot (1 - \alpha_1 t) \cdots (1 - \alpha_k t)
=  (1 - q^{i_1} \alpha_1 t) \cdots (1 - q^{i_k} \alpha_k t).
\end{equation}
%({\color{red}It seems to me easier to see~\eqref{cancelling-expression} specializes to the right hand side of the above equation. (Jia)})
For example, the expression~\eqref{cancelling-expression} with $i=3$ evaluated at $\varphi(5  \mid 1 4 6 \mid 2 3) = (\alpha_2, \alpha_3, q \cdot \alpha_3, q \cdot \alpha_2, \alpha_1, q^2 \cdot \alpha_2)$  equals
\[ \frac{(1-q\alpha_2 t)(\cancel{1-q\alpha_3 t})(1-q^2\alpha_3 t)}{(\cancel{1-\alpha_2 t})(\cancel{1-\alpha_3 t})(\cancel{1-q\alpha_3t})} (1-\alpha_1t)(\cancel{1-\alpha_2t})(\cancel{1-\alpha_3t}).
\]
Taking the coefficient of $t^k$ in this polynomial gives 
\[
(-1)^k \cdot q^{i_1 + \cdots + i_k} \cdot \alpha_1 \cdots \alpha_k = (-1)^k \cdot q^i \cdot \alpha_1 \cdots \alpha_k,
\]
proving the assertion (\ref{i-in-t-main-identity}).

As an immediate consequence of (\ref{i-in-t-main-identity}) we have
\begin{equation}
(-1)^{k+1} \cdot q^i \cdot \alpha_1 \cdots \alpha_k + 
\sum_{j = 0}^k (-1)^{k-j} (1-q) P_j(x_1, \dots, x_i; q) e_j(\alpha_1, \dots, \alpha_k) \in \II(Y_{n,k}^{(q)}).
\end{equation}
Taking the highest degree component gives
\begin{equation}
(1-q) P_k(x_1, \dots, x_i; q) \in \TT(Y_{n,k}^{(q)}).
\end{equation}
Since we are working over $F = \QQ(q)$, we have $1 - q \neq 0$ so that
 $P_k(x_1, \dots, x_i; q) \in \TT(Y_{n,k}^{(q)})$.  
 
So far we have demonstrated the inclusion $I_{n,k}^{(q)} \subseteq \TT(Y_{n,k}^{(q)})$.  
On the other hand, we have 
\[
\dim(F[\xx_n]/\TT(Y_{n,k}^{(q)}) = |Y_{n,k}^{(q)}| = |\OP_{n,k}| = \dim(F[\xx_n]/I_{n,k}^{(q)}),
\]
where the last equality follows from Corollary~\ref{ungraded-frobenius-image}. 
This forces $I_{n,k}^{(q)} = \TT(Y_{n,k}^{(q)})$.
\end{proof}

\section{Acknowledgements}
\label{Acknowledgements}

B. Rhoades was partially supported by NSF Grant DMS-1500838.
T. Scrimshaw was partially supported by NSF Grant
DMS-1148634.
This work benefited from computations using {\sc SageMath}~\cite{sage}.


\begin{thebibliography}{99}
 
 \bibitem{ABR} R. Adin, F. Brenti, and Y. Roichman.  Descent representations and multivariate statisitcs.
 {\it Trans. Amer. Math. Soc.}, {\bf 357} (2005), 3051--3082.

\bibitem{HeckeCoinv} 
R. M. Adin, A. Postnikov\ and\ Y. Roichman, Hecke algebra actions on the coinvariant algebra, {\it J. Alg.} {\bf 233} (2000), no.~2, 594--613. %MR1793918

\bibitem{GelfandModels}
R. M. Adin, A. Postnikov\ and\ Y. Roichman. Combinatorial Gelfand Models, {\it J. Alg.} {\bf 320} (2008), no.~3, 
1311--1325.
 
 %\bibitem{ARR}  D. Armstrong, V. Reiner, and B. Rhoades.  Parking spaces.
 %{\it Adv. Math.}, {\bf 269} (2015), 647--706.
 
 %\bibitem{Artin}  E. Artin.  {\it Galois Theory,} Second edition.
 %Notre Dame Math Lectures, no. 2.  Notre Dame: University of Notre Dame, 1944.
 
 %\bibitem{Bergeron}  F. Bergeron.  {\it Algebraic Combinatorics and Coinvariant Spaces.}
 %CMS Treatises in Mathematics.
%Boca Raton:  Taylor and Francis, 2009.

%\bibitem{BR}  A. Berget and B. Rhoades.  Extending the parking space.  {\it J. Comb. Theory, Ser. A,}
%{\bf 123 (1)} (2014), 43--56.

\bibitem{BumpHecke} 
Daniel Bump, Hecke Algebras, lecture notes retrieved from \url{http://sporadic.stanford.edu/bump/math263/hecke.pdf}.
 
 \bibitem{CR}  K.-T. J. Chan and B. Rhoades.  Generalized coinvariant algebras for wreath products.
Submitted, 2017. \arxiv{1701.06256}.
 
% \bibitem{C}  C. Chevalley.  Invariants of finite groups generated by reflections.
% {\it Amer. J. Math.}, {\bf 77 (4)} (1955), 778--782.
 
 \bibitem{GarsiaProcesi}  A. M. Garsia and C. Procesi.  On certain graded $S_n$-modules and the $q$-Kostka
 polynomials.  {\it Adv. Math.}, {\bf 94 (1)} (1992), 82--138.
 
 %\bibitem{GS}  A. M. Garsia and D. Stanton.  Group actions on Stanley-Reisner rings and invariants of permutation
 %groups.  {\it Adv. Math.}, {\bf 51 (2)} (1984), 107--201.
 
 \bibitem{GeckPfeiffer}  M. Geck and G. Pfeiffer.  {\it Characters of Finite Coxeter Groups and 
 Iwahori-Hecke Algebras}.  
 London Mathematical Society Monographs.
 Clarendon Press: Oxford, 2000.
 
% \bibitem{HHL}  J. Haglund, M. Haiman, and N. Loehr.  A combinatorial formula for nonsymmetric 
% Macdonald polynomials.  {\it Amer. J. of Math.}, {\bf 103} (2008), 359--383.
 
%  \bibitem{HHLRU} J. Haglund, M. Haiman, N. Loehr, J. B. Remmel, and A. Ulyanov.  A combinatorial
% formula for the character of the diagonal coinvariants. {\it Duke Math. J.} {\bf 126} (2005), pp. 195--232.
 
% \bibitem{HLMV}  J. Haglund, K. Luoto, S. Mason, and S. van Willigenburg.
% Refinements of the Littlewood-Richardson Rule.
% {\it Trans. Amer. ath. Soc.}, {\bf 363} (2011), 1665--1686.

\bibitem{HRW}  J. Haglund, J. Remmel, and A. T. Wilson.  The Delta Conjecture.  
Accepted, {\it Trans. Amer. Math. Soc.}, 2016.  \arxiv{1509.07058}.

\bibitem{HRS} J. Haglund, B. Rhoades, and M. Shimozono.  Ordered set partitions, generalized 
coinvariant algebras, and the Delta Conjecture.  
Accepted, {\it Adv. Math.}, 2018.
\arxiv{1609.07575}.


%\bibitem{HaimanConj}  M. Haiman.   Conjectures on the quotient ring by diagonal invariants.
%{\it J. Algebraic Combin.} {\bf 3 (1)} (1994),  17--76

%\bibitem{Haiman} M. Haiman.  Vanishing theorems and character formulas for the Hilbert scheme of points in the plane.
%{\it Invent. Math.}, {\bf 149 (2)}  (2002), 371--407.

%\bibitem{Huang}  J. Huang.  $0$-Hecke actions on coinvariants and flags.
%{\it J. Algebraic Combin.} {\bf 40} (2014) 245--278.

\bibitem{HR}  J. Huang and B. Rhoades.  Ordered set partitions and the 0-Hecke algebra.
{\it ALCO}, {\bf 1 (1)} (2018), 47--80.



\bibitem{Macdonald}  I. G. Macdonald.  {\it Symmetric Functions and Hall Polynomials,} Second edition.
Oxford Mathematican Monographs.
New York: The Clarendon Press Oxford University Press, 1995.
With contributions by A. Zelevinsky, Oxford Science Publications.

\bibitem{Meyer}  K. Meyer.  Descent representations of generalized coinvariant algebras.
Submitted, 2018.  \arxiv{1711.11355}.

%\bibitem{PS}  A. Postnikov and B. Shapiro.  Trees, parking functions, syzygies, and deformations of monomial
%ideals.  {\it Trans. Amer. Math. Soc.}, {\bf 356} (2004), 3109--3142.

%\bibitem{RW}  J. Remmel and A. T. Wilson.  An extension of MacMahon's Equidistribution
%Theorem to ordered set partitions.
%{\it J. Combin. Theory Ser. A}, {\bf 134} (2015), 242--277.

%\bibitem{RhoadesWilson}  B. Rhoades and A. T. Wilson.  Tail positive words and generalized coinvariant
%algebras.  Submitted, 2017.  {\tt arXiv:1705.02618}.

%\bibitem{Rhoades}  B. Rhoades.  Ordered set partition  statistics and the Delta Conjecture.
%Submitted, 2016. {\tt arXiv:1605.04007}.

\bibitem{sage} SageMath, the Sage Mathematics Software System (Version 8.0),
   The Sage Developers, 2017, \url{http://www.sagemath.org}.

%\bibitem{ST}  G. C. Shephard and J. A. Todd.  Finite unitary reflection groups.  {\it Can. J. Math.},
%{\bf 6} (1954), 274--304.



%\bibitem{Stanley}  R. P. Stanley.  Invariants of finite groups and their applications to combinatorics.
%{\it Bull. Amer. Math. Soc.}, {\bf 1} (1979), 475--511.

%\bibitem{Stein} E. Steingr\'imsson.  Statistics on Ordered Partitions of Sets.  Preprint, 2014.
%{\tt arXiv:0605670}.


%\bibitem{WMultiset}  A. T. Wilson.  An extension of MacMahon's Equidistribution Theorem
%to ordered multiset partitions. 
%{\it Electron. J. Combin.}, {\bf 23 (1)} (2016), P1.5.


  
\end{thebibliography}
\end{document}